\documentclass[10pt]{amsart}
\usepackage{amssymb}
\usepackage{amsbsy}
\usepackage{amscd}
\usepackage[mathscr]{eucal}
\usepackage{verbatim}
\usepackage{version}
\usepackage{color}
\usepackage[matrix,arrow]{xy}
\usepackage{graphicx}

\def\cal{\mathcal}
\def\Bbb{\mathbb}

\newenvironment{NB}{
\color{red}{\bf NB}. \footnotesize 
}{}
\excludeversion{NB}
\newenvironment{NB2}{
\color{blue}{\bf NB}. \footnotesize
}{}

\newcommand{\GL}  {\operatorname{GL}}

\newcommand{\SL}  {\operatorname{SL}}

\newcommand{\rk}{\operatorname{rk}}

\newcommand{\NS}{\operatorname{NS}}

\newcommand{\Pic}{\operatorname{Pic}}
\newcommand{\ch}{\operatorname{ch}}
\newcommand{\td}{\operatorname{td}}

\newcommand{\Coh}{\operatorname{Coh}}

\newcommand{\Amp}{\operatorname{Amp}}

\newcommand{\alg}{\operatorname{alg}}

\newcommand{\Stab}{\operatorname{Stab}}

\newcommand{\Mod}{M}

\font\b=cmr10 scaled \magstep5
\def\bigzerou{\smash{\lower1.7ex\hbox{\b 0}}}

\numberwithin{equation}{section}
\theoremstyle{plain}
 \newtheorem{thm}{Theorem}[section]
 \newtheorem{lem}[thm]{Lemma}
 \newtheorem{prop}[thm]{Proposition}

\theoremstyle{definition}
 \newtheorem{defn}[thm]{Definition}
 
\theoremstyle{remark}
 \newtheorem{rem}[thm]{Remark}

\begin{document}

\title{A note on stability conditions on 
an elliptic surface.}
\author{K\={o}ta Yoshioka}
\address{Department of Mathematics, Faculty of Science,
Kobe University,
Kobe, 657, Japan
}
\email{yoshioka@math.kobe-u.ac.jp}

\thanks{
The author is supported by the Grant-in-aid for 
Scientific Research (No. 18H01113), JSPS}

\begin{abstract}
We shall study stability conditions and Fourier-Mukai transforms on an elliptic surface. 
In particular we shall explain duality of elliptic surfaces by Fourier-Mukai transforms.
\end{abstract}

\maketitle

\renewcommand{\thefootnote}{\fnsymbol{footnote}}

\section{Introduction}

For an abelian variety $X$,
we have a notion of dual $\widehat{X}$.
It is an abelian variety
which is the moduli space $\Pic^0(X)$ of numerically trivial line bundles
on $X$, and $X$ is the moduli space of
numerically trivial line bundles on $\widehat{X}$
via the universal family ${\cal P}$:
We have an isomorphism 
$X \to \Pic^0(\widehat{X})$ by sending $x \in X$ to
${\cal P}_{|\{x \} \times \widehat{X}}$.
We call this property Fourier-Mukai duality in \cite{Y:duality}, since it was a key property
of Mukai's construction of
an equivalence
${\bf D}(X) \to {\bf D}(\widehat{X})$ \cite{Mukai:1981}.
It is now known that 
Fourier-Mukai duality holds for many moduli spaces $X'$ of stable sheaves
on abelian surfaces, K3 surfaces and elliptic surfaces $X$, if $\dim X=\dim X'$ \cite{BBH:1},
\cite{Mu:5}, \cite{PerverseII}.
In \cite{Y:duality}, we explained Fourier-Mukai duality on K3 surfaces by using Bridgeland stability conditions.
 In this note, by introducing a stability condition,  
we shall give a similar explanation for an elliptic surface
(Theorem \ref{thm:rk=0}).
our argument is the same as in \cite{Y:duality}.
As an application, we shall recover a result in \cite{Y:7} that
a suitable twisted stability for a semi-stable 1-dimensional sheaf
is preserved under a Fourier-Mukai transform
by using Bridgeland stability condition (see Proposition \ref{prop:r=0}). 
As another motivation of this note,
we discuss the relation of walls for Bridgeland stability and $\lambda$-stability
in \cite{Y:elliptic} when $X$ is an elliptic K3 surface. 

Let us briefly explain the organization of this note.
In sections \ref{subsect:stability} and \ref{subsect:geometric},
we recall Bridgeland stability conditions.
In section \ref{subsect:ellipticMukai},
we introduce a symmetric pairing on $H^*(X,{\Bbb Q})$ 
which is preserved under Fourier-Mukai transforms. 
By using our pairing, we introduce a stability condition, and show that it is equivalent to
a geometric stability condition constructed by Arcara and Bertram \cite{AB} (Proposition \ref{prop:hat}).
We next study the Fourier-Mukai transform of our stability condition in section \ref{sect:FM}, and prove that
the obtained stability condition is equivalent to 
a geometric stability condition (Proposition \ref{prop:FMstability}). 
Then Theorem \ref{thm:rk=0} easily follows.
All these arguments are the same as in \cite{Y:duality}.
In section \ref{sect:stability} and section \ref{sect:K3}, 
we give applications.

For related results, we would like to remark that Lo and Martinez \cite{Lo-M} obtained almost all results in this note
for the Fourier-Mukai transform between a Weierstrass elliptic surface and 
its compactfied relative Picard scheme.

\section{Stability conditions and Fourier-Mukai transforms on an elliptic surface.}

\subsection{Stability conditions}\label{subsect:stability}
Let us briefly recall some notation on stability conditions.
For more details, see \cite{Br:stability}, \cite{Br:3} and \cite{Br:4}.
A stability condition $\sigma=(Z_\sigma,{\cal P}_\sigma)$ 
on ${\bf D}(X)$
consists of a group homomorphism $Z_\sigma: {\bf D}(X) \to {\Bbb C}$
and a slicing ${\cal P}_\sigma$ of ${\bf D}(X)$ such that if
$0 \ne E \in {\cal P}_\sigma(\phi)$ then 
$Z_\sigma(E) =m(E)\exp(\pi \sqrt{-1} \phi)$ 
for some $m(E) \in {\Bbb R}_{>0}$.
The set of stability conditions has a structure of complex 
manifold.
We denote this space by $\Stab(X)$.

\begin{defn}\label{defn:B-stable}
A $\sigma$-semi-stable object $E$ of phase $\phi$ is
an object of  ${\cal P}_\sigma(\phi)$.
If $E$ is a simple object of ${\cal P}_\sigma(\phi)$, then
$E$ is $\sigma$-stable. 
\end{defn}

By \cite[Prop. 5.3]{Br:stability},
giving a stability condition $\sigma$ is
the same as giving a bounded $t$-structure on ${\bf D}(X)$ 
and a stability function $Z_\sigma$ on its heart ${\cal A}_\sigma$
with the Harder-Narasimhan property.
For $\sigma=(Z_\sigma,{\cal P}_\sigma)$,
we have the relation ${\cal A}_\sigma ={\cal P}_\sigma((0,1])$,
where ${\cal P}_\sigma((0,1])$ is the subcategory of ${\bf D}(X)$
generated by semi-stable objects $E \in {\cal P}_\sigma(\phi)$ with 
$\phi \in (0,1]$.  
Since the pair $(Z_\sigma,{\cal A}_\sigma)$
defines a stability condition, 
we also use the symbol $\sigma=(Z_\sigma,{\cal A}_\sigma)$
to denote a stability condition.
For a $\sigma$-semi-stable object $E \in 
{\cal P}_\sigma(\phi)$, we set $\phi_\sigma(E)=\phi$.
For $E \in {\cal P}_\sigma((0,1])$,
we also set $\phi_\sigma(E) \in (0,1]$ 
by $Z_\sigma(E)=m(E)\exp(\pi \sqrt{-1} \phi_\sigma(E))$. 

For an equivalence $\Phi:{\bf D}(X) \to {\bf D}(X')$,
we have an isomorphism $\Phi:\Stab(X) \to \Stab(X')$
such that 
$\Phi(\sigma)$ $(\sigma \in \Stab(X))$ 
is a stability condition given by
\begin{equation}\label{eq:FM-action}
\begin{split}
Z_{\Phi(\sigma)}=& Z_\sigma \circ \Phi^{-1}:{\bf D}(X') \to {\Bbb C},\\
{\cal P}_{\Phi(\sigma)}(\phi)=& \Phi({\cal P}_\sigma(\phi)).
\end{split}
\end{equation}
We also have an action of the universal covering 
$\widetilde{\GL}_2^+({\Bbb R})$ of $\GL_2^+({\Bbb R})$ on 
$\Stab(X)$.
Since ${\Bbb C}^{\times} \subset  \GL_2^+({\Bbb R})$,
we have an injective homomorphism
${\Bbb C} \to  \widetilde{\GL}_2^+({\Bbb R})$.
Thus 
we have an action of $\lambda \in {\Bbb C}$ on
$\Stab(X)$. 
For a stability condition $\sigma \in \Stab(X)$, 
$\lambda(\sigma)$ is 
given by
\begin{equation}\label{eq:C-action}
\begin{split}
Z_{\lambda(\sigma)}=&
\exp(-\pi  \sqrt{-1} \lambda)Z_\sigma\\
{\cal P}_{\lambda(\sigma)}(\phi)=&
{\cal P}_\sigma(\phi+\mathrm{Re}\lambda).
\end{split}
\end{equation}

\subsection{Geometric stability conditions}\label{subsect:geometric}
Bridgeland \cite{Br:3} and Arcara and Bertram \cite{AB}
constructed a stability condition
$\sigma_{(\beta,\omega)}=(Z_{(\beta,\omega)},{\cal A}_{(\beta,\omega)})$
which is characterized
by the stability of $k_x$ ($x \in X$), where
$Z_{(\beta,\omega)}(\bullet)=
\langle e^{\beta+\sqrt{-1}\omega},\ch(\bullet) \rangle:
{\bf D}(X) \to {\Bbb C}$ is the stability function and
${\cal A}_{(\beta,\omega)}$ 
is an abelian category which is a tilting of $\Coh(X)$
by a torsion pair $({\cal T}_{(\beta,\omega)},{\cal F}_{(\beta,\omega)})$:
\begin{enumerate}
\item
${\cal T}_{(\beta,\omega)}$ is a full subcategory of $\Coh(X)$ generated by
torsion sheaves and $\mu$-stable sheaves $E$ with 
$(c_1(E)-\rk E \beta,\omega)>0$ and
\item
${\cal F}_{(\beta,\omega)}$ is a full subcategory of $\Coh(X)$ generated by
$\mu$-stable sheaves $E$ with 
$(c_1(E)-\rk E \beta,\omega) \leq 0$.
\end{enumerate}
This kind of stability conditions are called geometric.
To be more precise,
for a geometric stability condition
$\sigma=(Z_\sigma,{\cal A}_\sigma)$,
we require that
$\mathrm{Re}\mho, \mathrm{Im}\mho$ span a positive definite
2-plane of $H^*(X,{\Bbb R})$,
where $Z_\sigma(\bullet)=\langle \mho,\bullet \rangle$
with $\mho \in H^*(X,{\Bbb Q})_{\alg} \otimes {\Bbb C}$. 
Up to the action of $\widetilde{\GL}_2^+({\Bbb R})$, 
there is $(\beta,\omega)$ such that
$\sigma=\sigma_{(\beta,\omega)}$.
We shall use this fact for a stability condition
on an elliptic surface (see subsection \ref{subsect:ellipticStab} in particular Lemma \ref{lem:geometric}). 
\begin{NB}
Old:
$\sigma \in \Stab(X)$ satisfies  
(1)
$Z_\sigma(\bullet)=\langle \mho,\bullet \rangle$
with $\langle \mathrm{Re}\mho^2 \rangle=
\langle \mathrm{Im}\mho^2 \rangle$, 
$\langle \mathrm{Re}\mho,\mathrm{Im}\mho \rangle=0$
and
(2)
$k_x$ ($x \in X$) are $\sigma$-stable
with the phase $\phi_{\sigma}(k_x)=1$. Then
$\sigma=\sigma_{(\beta,\omega)}$ for some $(\beta,\omega)$.
\end{NB}
\begin{NB}
\begin{rem}
In order to construct a geometric stability condition $\sigma$
with $Z_\sigma=Z_{(\beta,\omega)}$,
we require that
$$
(\beta, \omega) \in 
\NS(X)_{\Bbb R} \times \Amp(X)_{\Bbb R} \setminus \bigcup_{v_1 \in \Delta(X)}
\left\{(\beta,\omega) \left| \langle e^{\beta+\sqrt{-1} \omega},v_1 \rangle 
\in {\Bbb R}_{\leq 0} \right. \right\},
$$
where $\Delta(X)$ is the set of  $(-2)$-vector of $H^*(X,{\Bbb Z})_{\alg}$.
\end{rem}
\end{NB}

\subsection{Mukai pairing and Fourier-Mukai transform on an elliptic surface.}\label{subsect:ellipticMukai} 

Let $\pi:X \to C$ be an elliptic surface over a curve $C$
and $f$ a fiber of $\pi$.
Let 
$$
H^{*}(X,{\Bbb Q})_{\alg}:=H^0(X,{\Bbb Q}) \oplus \NS(X)_{\Bbb Q} \oplus H^4(X,{\Bbb Q})
$$
be the algebraic part of the cohomology ring $H^*(X,{\Bbb Q})$.
We define $\varrho_X \in H^4(X,{\Bbb Z})$ by
$\int_X \varrho_X=1$.
Then we have
$$
H^{*}(X,{\Bbb Q})_{\alg}={\Bbb Q} \oplus \NS(X)_{\Bbb Q} \oplus {\Bbb Q}\varrho_X.
$$ 
For $u=x_0+x_1+x_2 \varrho_X$ ($x_0,x_2 \in {\Bbb Q}$, $x_1 \in \NS(X)_{\Bbb Q}$),
we set $u^*:=x_0-x_1+x_2 \varrho_X$.

\begin{defn}
\begin{enumerate}
\item[(1)]
We define the Mukai pairing $\langle \;\;,\;\; \rangle$ on 
$H^*(X,{\Bbb Q})_{\alg}$
by
\begin{equation}
\langle u,v \rangle=\int_X u^* v=(x_1 \cdot y_1)-x_0 y_2-x_2 y_0,\;
u=x_0+x_1+x_2 \varrho_X, v=y_0+y_1+y_2 \varrho_X,
\end{equation} 
where $(x_1 \cdot y_1)$ is the intersection pairing of $\NS(X)_{\Bbb Q}$.
\item[(2)]
For $E \in {\bf D}(X)$, we define the Mukai vector of $E$ by
\begin{equation}
v(E):=\ch E (1+\tfrac{\chi({\cal O}_X)}{2}\varrho_X) \in H^*(X,{\Bbb Q})_{\alg}.
\end{equation}
\end{enumerate}
\end{defn}

\begin{lem}
For $E, F \in {\bf D}(X)$, 
\begin{equation}
-\langle v(E),v(F) \rangle=\chi(E,F(\tfrac{1}{2}K_X)).
\end{equation}
\end{lem}

\begin{proof}
\begin{equation}
\chi(E,F(\tfrac{1}{2}K_X))=\int_X (\ch E)^* \ch F e^{\frac{1}{2}K_X}\td_X=
\int_X v(E)^* v(F)=-\langle v(E),v(F) \rangle.
\end{equation}
\end{proof}

\begin{defn}\label{defn:moduli}
Let $v \in H^*(X,{\Bbb Q})_{\alg}$ be a Mukai vector (of a coherent sheaf).
\begin{enumerate}
\item[(1)]
For an ample divisor $H$ on $X$ and $\beta \in \NS(X)_{\Bbb Q}$,
let $\Mod_H^\beta(v)$ be the moduli space of
$S$-equivalence classes of 
$\beta$-twisted semi-stable sheaves $E$ with
the Mukai vector $v(E)=v$
and ${\cal M}_H^\beta(v)^{ss}$ the moduli stack of
$\beta$-twisted  semi-stable sheaves $E$ with $v(E)=v$. 
\item[(2)]
For a stability condition $\sigma$,
we denote the moduli stack of $\sigma$-semi-stable objects $E$ with 
$v(E)=v$ by ${\cal M}_\sigma(v)$,
where we usually choose 
$\phi:=\mathrm{Im}(\log Z_\sigma(v))/\pi \in (-1,1]$.
We also denote the moduli scheme by $M_\sigma(v)$ if it exists.
\end{enumerate}
\end{defn}

\begin{NB}
\begin{defn}
Let $\sigma$ be a stability condition.
For $v \in H^*(X,{\Bbb Q})_{\alg}$,
we denote the moduli stack of $\sigma$-semi-stable objects $E$ with 
$v(E)=v$ by ${\cal M}_\sigma(v)$,
where we usually choose 
$\phi:=\mathrm{Im}(\log Z_\sigma(v))/\pi \in (-1,1]$.
\end{defn}
\end{NB}

$M_H^\beta(v)$ is a projective scheme.
If $H$ is general in the ample cone,
$\Mod_H^\beta(v)$ and ${\cal M}_H^\beta(v)^{ss}$ are
 independent of the choice of $\beta$.


\begin{defn}
For elliptic surfaces $X \to C$, $Y \to C$
and ${\bf P} \in {\bf D}(X \times Y)$, 
let $\Phi_{X \to Y}^{{\bf P}}:{\bf D}(X) \to {\bf D}(Y)$ be an integral functor
whose kernel is ${\bf P}$:
\begin{equation}
\Phi_{X \to Y}^{{\bf P}}(E)={\bf R}p_{Y*}(p_X^*(E) \otimes {\bf P}), \; \; E \in {\bf D}(X).
\end{equation}
We also set
\begin{equation}
v({\bf P}):=\ch ({\bf P})p_X^*((1+\tfrac{\chi({\cal O}_X)}{2}\varrho_X))
p_Y^*((1+\tfrac{\chi({\cal O}_Y)}{2}\varrho_Y)) \in H^*(X \times Y,{\Bbb Q}).
\end{equation}
\end{defn}

As a consequence of Grothendieck Riemann-Roch theorem, we have the following.
\begin{lem}
For $\Phi:=\Phi_{X \to Y}^{{\bf P}^{\vee}}$,
we have a commutative diagram

\begin{equation}
\begin{CD}
{\bf D}(X) @>{\Phi}>> {\bf D}(Y)\\
@V{v}VV @VV{v}V \\
H^*(X,{\Bbb Q})_{\alg} @>{\overline{\Phi}}>> H^*(Y,{\Bbb Q})_{\alg}
\end{CD}
\end{equation}
where 
\begin{equation}
\overline{\Phi}(\alpha):=p_{Y*}(v({\bf P}^{\vee})p_X^*(e^{-\tfrac{1}{2}K_X})p_X^*(\alpha)).
\end{equation}
\end{lem}

\begin{NB}
Serre duality implies
\begin{equation}
\begin{split}
\chi(E,F(\tfrac{1}{2}K_X))=&\chi(F(\tfrac{1}{2}K_X),E(K_X))=\chi(F,E(\tfrac{1}{2}K_X)).
\end{split}
\end{equation}
\end{NB}

\begin{lem}\label{lem:fiber}
Assume that $\Phi$ is a relative Fourier-Mukai transform, that is,
${\bf P}_{|X \times \{ y \}}$ $(y \in Y)$ is a stable 1-dimensional sheaf on a fiber of $\pi$. 
Then
for $E \in {\bf D}(X)$,
$\overline{\Phi}(v(E)e^{\lambda f})=e^{\lambda f}\overline{\Phi}(v(E))$.
\end{lem}

\begin{proof}
By
\begin{equation}
(\ch E)e^{\lambda f}=\ch E+\lambda \ch(E_{|f}),
\end{equation}
we get 
\begin{equation}
v(E)e^{\lambda f}=v(E)+\lambda v(E_{|f}).
\end{equation}
Then we see that

\begin{equation}
\begin{split}
v(\Phi(E))e^{\lambda f}=& v (\Phi(E))+\lambda v(\Phi(E)_{|f})\\
=& v(\Phi(E))+\lambda v(\Phi(E_{|f}))\\
=& \overline{\Phi}(v(E))+ \overline{\Phi}(\lambda v(E_{|f}))\\
=&\overline{\Phi}( v(E)e^{\lambda f}).
\end{split}
\end{equation}
\end{proof}


\begin{prop}\label{prop:isometry}
Let $\Phi$ be an equivalence in Lemma \ref{lem:fiber}. Then 
$\overline{\Phi}$ preserves
$\langle\;\;,\;\; \rangle$.
Thus 
$$
\langle \overline{\Phi}(v(E_1)),\overline{\Phi}(v(E_2)) \rangle=\langle v(E_1),v(E_2) \rangle,\;
E_1,E_2 \in {\bf D}(X).
$$ 
\end{prop}

\begin{proof}
Let $\widehat{\Phi}:{\bf D}(Y) \to {\bf D}(X)$ be a quasi-inverse of $\Phi$.
For $E \in {\bf D}(X)$ and $F \in {\bf D}(Y)$, by using Lemma \ref{lem:fiber},
we get 
\begin{equation}
\begin{split}
\chi(E,\widehat{\Phi}(F)(\tfrac{1}{2}K_X))=&\chi(E(-\tfrac{1}{2}K_X),\widehat{\Phi}(F))\\
=& \chi(\Phi(E(-\tfrac{1}{2}K_X)),F)\\
=& \chi(\Phi(E)(-\tfrac{1}{2}K_Y),F)\\
=&\chi(\Phi(E),F(\tfrac{1}{2}K_Y)).
\end{split}
\end{equation}
Hence our claim holds.
\end{proof}

\begin{NB}
\begin{lem}
$\chi(\Phi_{X \to Y}^{{\bf P}^* \otimes p_X^*(\omega_X)}(E),F)=
\chi(E, \Phi_{Y \to X}^{{\bf P}}(F))$.
\end{lem}

\begin{proof}
By the Serre duality, we get
\begin{equation}
\begin{split}
v(\Phi_{X \to Y}^{{\bf P}^* \otimes p_X^*(\omega_X)}(E))^*=& p_{Y*}(v({\bf P})^* p_X^*(v(E)e^{K_X})p_X^*(e^{-\frac{1}{2}K_X}))^*\\
=& p_{Y*}((v({\bf P}) p_X^*(v(E)^*)p_X^*(e^{-\frac{1}{2}K_X})).
\end{split}
\end{equation}

\begin{equation}
\begin{split}
\int_Y v(\Phi_{X \to Y}^{{\bf P}^* \otimes p_X^*(\omega_X)}(E))^* v(F)e^{-\frac{1}{2}K_Y}=&
\int_Y p_{Y*}((v({\bf P}) p_X^*(v(E)^*)p_X^*(e^{-\frac{1}{2}K_X})) v(F) e^{-\frac{1}{2}K_Y}\\
=& \int_{X \times Y} v({\bf P}) p_X^*(v(E)^*)p_X^*(e^{-\frac{1}{2}K_X}) 
p_Y^*(v(F) e^{-\frac{1}{2}K_Y})\\
=& \int_X p_{X*}(v({\bf P}) p_X^*(v(E)^*)p_X^*(e^{-\frac{1}{2}K_X}) 
p_Y^*(v(F) e^{-\frac{1}{2}K_Y}))\\
=& \int_X v(E)^* e^{-\frac{1}{2}K_X}  p_{X*}(v({\bf P}) 
p_Y^*(v(F) e^{-\frac{1}{2}K_Y}))\\
=& \int_X v(E)^* e^{-\frac{1}{2}K_X} v(\Phi_{Y \to X}^{{\bf P}}(F))
\end{split}
\end{equation}
\end{proof}
\end{NB}

\subsection{A stability condition on an elliptic surface.}\label{subsect:ellipticStab} 

Let $\pi:X \to C$ be an elliptic surface over a curve $C$.
We shall introduce a stability condition which has a better behavior than 
the stability condition $\sigma_{(\beta, \omega)}$ in subsection \ref{subsect:geometric}
under the Fourier-Mukai transform $\Phi$.

\begin{defn}\label{defn:hat}
For $(\beta,\omega) \in \NS(X)_{\Bbb R} \times \Amp(X)_{\Bbb R}$ with 
$(\omega^2)>\chi({\cal O}_X)$, 
we set 
$$
\widehat{Z}_{(\beta,\omega)}(E):=\langle e^{\beta+\sqrt{-1} \omega}, v(E) \rangle,\; E \in {\bf D}(X)
$$
and set
\begin{equation}
\widehat{\sigma}_{(\beta,\omega)}:=(\widehat{Z}_{(\beta,\omega)},{\cal A}_{(\beta,\omega)}).
\end{equation}
\end{defn}

We shall prove that $\widehat{\sigma}_{(\beta,\omega)}$ is equivalent to the stability condition
$\sigma_{(\beta,\omega')}$ by the action of $\widetilde{\GL}_2^+({\Bbb R})$.
For $H \in \Amp(X)_{\Bbb R}$, we set $\omega:=tH$ and
$s:=\sqrt{t^2-\frac{\chi({\cal O}_X)}{(H^2)}}$, where $t^2-\frac{\chi({\cal O}_X)}{(H^2)}>0$.
Let $(T,f) \in \widetilde{\GL}_2^+({\Bbb R})$ be an element such that
$T \in \GL_2^+({\Bbb R})$ acts on ${\Bbb R}^2={\Bbb R}+{\Bbb R}\sqrt{-1}$ as 
\begin{equation}
T(x+y \sqrt{-1})=x+\tfrac{s}{t}y \sqrt{-1}
\end{equation}
and the increasing function $f:{\Bbb R} \to {\Bbb R}$ satisfies $f(n)=n$ for $n \in {\Bbb Z}$.
Then we get the following result.
\begin{prop}\label{prop:hat}
$(T,f) \cdot(Z_{(\beta,sH)},{\cal A}_{(\beta,sH)})=(\widehat{Z}_{(\beta,tH)},{\cal A}_{(\beta,tH)})$.
In particular $\widehat{\sigma}_{(\beta,\omega)}$ is a stability condition such that
$k_x$ $(x \in X)$ are stable objects of phase 1.
\end{prop}

\begin{proof}
We set $\sigma':=(T,f) \cdot \sigma_{(\beta,sH)}$.

\begin{equation}
\widehat{Z}_{(\beta,tH)}(E)=T^{-1} \circ Z_{(\beta,sH)}(E).
\end{equation}
We set
\begin{equation}
\ch(E)=e^\beta (r+\xi+a\varrho_X), r \in {\Bbb Z},\; \xi \in \NS(X)_{\Bbb R}, a \in {\Bbb R}.
\end{equation}
Then 
\begin{equation}
\begin{split}
\widehat{Z}_{(\beta,tH)}(E)= & \langle e^\beta(1+\sqrt{-1} tH-\tfrac{t^2(H^2)}{2}\varrho_X),
e^\beta (r+\xi+a\varrho_X)(1+\tfrac{\chi({\cal O}_X)}{2}\varrho_X) \rangle\\
=& \frac{r}{2}(t^2(H^2)-\chi({\cal O}_X))-a+\sqrt{-1} t(H \cdot \xi)\\
=& \frac{r}{2}s^2 (H^2) -a+\sqrt{-1} \frac{t}{s} (sH \cdot \xi)\\
=& T^{-1} \circ Z_{(\beta,sH)}.
\end{split}
\end{equation}
By the definition of $(T,f)$, we have ${\cal P}_{\sigma'}((0,1])={\cal P}_{\sigma_{(\beta,sH)}}((0,1])=
{\cal A}_{(\beta,sH)}={\cal A}_{(\beta,tH)}$.
Hence the claim holds.
\end{proof}

The following claim is proved in \cite[Prop. 10.3]{Br:3}.
\begin{lem}\label{lem:geometric}
Let $H$ be a ${\Bbb Q}$-divisor with $(H^2)>0$.
Let $\sigma=(\langle e^{\beta+\sqrt{-1}tH}, v(\bullet) \rangle,{\cal A})$
be a stability condition such that $k_x$ $(x \in X)$ are stable objects of phase 1.
Then ${\cal A}={\cal A}_{(\beta,tH)}$.
\end{lem}

\begin{prop}[{\cite[Prop. 3.2.1]{MYY:2011:1}, \cite[Prop. 6.11]{Y:blowup}}]\label{prop:asymptotic}
Let $L$ be an ample ${\Bbb Q}$-divisor.
For a Mukai vector $v=e^\beta(r+\xi+a \varrho_X)$, we
decompose 
$\xi$ as 
$$
\xi=dL+D,\; d=\frac{(\xi \cdot L)}{(L^2)},\; D \in L^\perp.
$$
We set  $\delta:=\min\{(D \cdot L)>0 \mid D \in \NS(X) \}/(L^2)$.
Then the following claims hold.
\begin{enumerate}
\item[(1)]
Assume that $r>0$.
Then 
${\cal M}_{\widehat{\sigma}_{(\beta,tL)}}(v)={\cal M}_L^{\beta-\frac{1}{2}K_X}(v)^{ss}$
if
\begin{equation}
\frac{t^2(L^2)-\chi({\cal O}_X)}{2}>
\frac{d}{\delta}
\left(d^2\frac{(L^2)}{2}-ra +r^2 \frac{\chi({\cal O}_X)}{2} \right).
\end{equation}
\item[(2)]
Assume that $r=0$. Then 
${\cal M}_{\widehat{\sigma}_{(\beta,tL)}}(v)={\cal M}_L^{\beta-\frac{1}{2}K_X}(v)^{ss}$
if
\begin{equation}
\frac{t^2(L^2)-\chi({\cal O}_X)}{2}>
\frac{d}{\delta}
\left(|a|+d^2(L^2) \right).
\end{equation}
\end{enumerate}
\end{prop}

\begin{proof}
We note that
${\cal M}_{\widehat{\sigma}_{(\beta,tL)}}(v)={\cal M}_{\sigma_{(\beta,sL)}}(v)$
by Proposition \ref{prop:hat},
where $s^2 (L^2)=t^2 (L^2)-\chi({\cal O}_X)$.
By the definition of $v(E)$,
$\langle v(E)^2 \rangle=\langle \ch(E)^2 \rangle-r^2 \chi({\cal O}_X)$.
Then 
\cite[Prop. 6.11]{Y:blowup} implies the claim.
\end{proof}

\begin{rem}\label{rem:delta}
If $l L \in \NS(X)$, then $\delta \geq \frac{1}{l}$.
\end{rem}

\begin{NB}
\subsection{}

If there is a section, then the wall-crossing problem is reduced to 
the wall-crossing behavior of twisted stability for 1-dimensional sheaves.

It is necessary to relate the birational equivalence classes and
the wall crossing.

\begin{equation}
\begin{split}
-\langle \ch F,\ch E \rangle
=& \int_X (\ch F)^{\vee} \ch E \\
=&
\int_X \ch (F(-\tfrac{1}{2}K_X))^{\vee} \ch E (\td_X-\chi({\cal O}_X))\\
=& \chi(F(-\tfrac{1}{2}K_X),E)-\rk E \chi({\cal O}_X)
\end{split}
\end{equation}

\begin{equation}
\begin{split}
-\langle e^{\beta+\sqrt{-1}\omega},\ch E \rangle
=& \int_X e^{-(\beta+\sqrt{-1}\omega} \ch E \\
=&
\int_X e^{-(\beta-\frac{1}{2}K_X +\sqrt{-1} \omega)}\ch E (\td_X-\chi({\cal O}_X))\\
=& \chi(e^{\beta-\frac{1}{2}K_X +\sqrt{-1} \omega},E)-\rk E \chi({\cal O}_X)
\end{split}
\end{equation}

\begin{equation}
\begin{split}
-\langle e^{\beta+\sqrt{-1}\omega},\ch \Phi(E) \rangle
=& \chi(e^{\beta-\frac{1}{2}K_X +\sqrt{-1} \omega},\Phi(E))-\rk \Phi(E) \chi({\cal O}_X)\\
=& \chi(\Phi(e^{\beta-\frac{1}{2}K_X +\sqrt{-1} \omega}),E)-\rk \Phi(E) \chi({\cal O}_X)\\
=& -\langle \Phi(e^{\beta-\frac{1}{2}K_X +\sqrt{-1} \omega})(\tfrac{1}{2}K_X),\ch E \rangle
+(\rk E \rk \Phi(e^{\beta-\frac{1}{2}K_X +\sqrt{-1} \omega})-\rk \Phi(E) \chi({\cal O}_X)
\end{split}
\end{equation}

If $X=C \times D$, then 
$\chi({\cal O}_X)=0$ and $\Phi(E(-\tfrac{1}{2}K_X))=\Phi(E)(-\frac{1}{2}K_X)$.
Hence $\Phi$ is an isometry.
\end{NB}

\section{Fourier-Mukai duality}\label{sect:FM}

\subsection{A Fourier-Mukai transform on an elliptic surface.}

Let $\pi:X \to C$ be an elliptic surface over a curve $C$
and $f$ a fiber of $\pi$.
Let $H$ be a relatively ample ${\Bbb Q}$-divisor such that
$(H \cdot f)=1$.
Let $v_0=\eta+b \varrho_X$ be a primitive and isotropic Mukai vector 
such that $\eta=r_0 f$, that is, $\eta$ is a nef and effective divisor with
$(\eta^2)=0$.  
Replacing $H$ by $H+x f$, we may assume that
$(H^2)=0$.
We shall consider $\beta$-twisted stability with respect to
$H_n:=H+nf$ $(n \gg 0)$.
We can take $e^\beta$ such that $\langle e^\beta,v_0 \rangle=0$.
Indeed we can choose $y$ 
satisfying
$\langle e^{\beta+yH_n},v_0 \rangle=\langle e^\beta,v_0 \rangle-y(H \cdot \eta)=0$. 
Then $v_0=\eta e^\beta$.
Assume that $X':=\Mod_{H_n}^\beta(v_0)$ is a fine moduli
scheme consisting of $\beta$-twisted stable sheaves. Then
$X'$ is an elliptic surface with a fibration 
$\pi':X' \to C$. We denote a fiber of $X' \to C$ by $f'$.

\begin{lem}
${\cal M}_{H+nf}^\beta(v_0)^{ss}$ depends only on
$H$ and
$\beta \mod {\Bbb Q}f+{\Bbb Q}H$.
In particular ${\cal M}_{H+nf}^{\beta-\frac{1}{2}K_X}(v_0)^{ss}={\cal M}_{H+nf}^\beta(v_0)^{ss}$.
\end{lem}

\begin{proof}
Let $E$ be a purely 1-dimensional sheaf with $v(E)=v_0$.
For an exact sequence
$$
0 \to E_1 \to E \to E_2 \to 0,
$$

$$
0=(c_1(E) \cdot f)=(c_1(E_1) \cdot f)+(c_1(E_2) \cdot f)
$$
 implies
that $(c_1(E_1) \cdot f)=(c_1(E_2) \cdot f)=0$.
Hence $(c_1(E_i) \cdot (H+yf))$ $(i=1,2)$ is independent of $y$.
Since 
$$
\frac{\chi(e^{\beta+xH+yf},E_i)}{(c_1(E_i) \cdot (H+nf))}=
\frac{\chi(e^{\beta+xH+yf},E_i)}{(c_1(E_i) \cdot H)}=
\frac{\chi(e^\beta,E_i)}{(c_1(E_i) \cdot H)}-x,
$$
the twisted semi-stability is independent of the choice of $x,y$.
\end{proof}

For a universal family ${\cal E}$,
we have a Fourier-Mukai transform
$\Phi:=\Phi_{X \to X'}^{{\cal E}^{\vee}[2]}$. 
We set $v_0':=v({\cal E}_{|\{ x \} \times X'}^{\vee}[1])=
-\overline{\Phi}(\varrho_X)$ and $H':=c_1(\overline{\Phi}(r_0 e^\beta))$.
Then $c_1(v_0')=r_0 f'$, $({H'}^2)=0$ and $(H' \cdot f')=1$.
\begin{NB}
$r_0 (H',f')=(H',c_1(v_0'))=\langle r_0 e^\beta,-\varrho_X \rangle=r_0$.
\end{NB}
Since $\langle \overline{\Phi}(He^\beta)^2 \rangle=0$ and
$\langle \overline{\Phi}(He^\beta),\varrho_{X'} \rangle=(H \cdot \eta)=r_0$,
there is $\beta' \in \NS(X')_{\Bbb Q}$ with
$\overline{\Phi}(He^\beta)=-r_0 e^{\beta'}$.
Then we have 
\begin{equation}\label{eq:H'}
\begin{split}
\overline{\Phi}(r_0 e^\beta)&=H' e^{\beta'}\\
\overline{\Phi}(He^\beta)&=-r_0 e^{\beta'}\\
\overline{\Phi}(v_0)&=\varrho_{X'}\\
\overline{\Phi}(\varrho_X)&=-v_0'.
\end{split}
\end{equation}
By $\langle \overline{\Phi}(He^\beta),\overline{\Phi}(\varrho_X) \rangle=0$, we get
$v_0'=e^{\beta'}(r_0 f')$.
\begin{NB}
$\Phi({\cal E}_{|X \times \{ x' \}})={\Bbb C}_{x'}$.
$\Phi({\Bbb C}_x)=({\cal E}_{|\{x \} \times X'}^{\vee}[1])[-1]$.
For a torsion free sheaf $E$ on an irreducible curve $C \in |nH|$,
$\Phi(E)[-1] \in \Coh(X')$.
\end{NB}
Hence we obtain the following lemma.
\begin{lem}\label{lem:coh-FM}
For 
$$
v=e^\beta(r+\xi+a \varrho_X)=e^\beta(r+pH+qf+D+a \varrho_X), D \in f^\perp \cap H^\perp,
$$
we have 
$$
\overline{\Phi}(v)=e^{\beta'}(\tfrac{r}{r_0}H'-pr_0+\frac{q}{r_0}\varrho_{X'}+D'-ar_0 f'),\; 
D' \in \NS(X') \cap {H'}^\perp \cap {f'}^\perp.
$$
\end{lem}

\subsection{A Fourier-Mukai transform and a stability condition.}

We take $l,l' \in {\Bbb Z}$ such that $lH \in \NS(X), l' H' \in \NS(X')$.
Since
\begin{equation}
e^{\beta+\sqrt{-1}m(H+nf)}=
e^\beta+\sqrt{-1}m( He^\beta+n f e^\beta)
-nm^2 \varrho_X,
\end{equation}
we have
\begin{equation}\label{eq:elliptic:Phi}
\begin{split}
\overline{\Phi}(e^{\beta+\sqrt{-1}m(H+nf)})
=& \frac{1}{r_0}H' e^{\beta'}+\sqrt{-1}m(-r_0 e^{\beta'}+(n/r_0) \varrho_{X'})
+nm^2 v_0'\\
=& -\sqrt{-1}r_0 m e^{\beta'+\sqrt{-1}\frac{1}{r_0^2m}(H'+r_0^2 m^2 nf')}.
\end{split}
\end{equation}
Then we have a generalization of \cite[2.5]{Y:duality}.
\begin{prop}\label{prop:FMstability}
Assume that $n>\frac{r_0^2}{2}\chi({\cal O}_X)$ and
$n>\tfrac{l r_0^3}{2n}+\frac{\chi({\cal O}_X)}{2}$. Then.
\begin{equation}
\Phi(\widehat{\sigma}_{(\beta,mH_n)}) =
\alpha( \widehat{\sigma}_{(\beta',\frac{1}{r_0^2 m}H'_{r_0^2 m^2 n})}),\;
\alpha=\frac{1}{2}+\frac{\log(r_0 m)}{\pi}\sqrt{-1},
\end{equation}
where we set $H_n':=H'+nf'$.
\begin{NB}
Old version
\begin{equation}
\Phi(\widehat{\sigma}_{(\beta,mH_n)}) =
r_0 m \exp \left(-\pi \frac{\sqrt{-1}}{2}\right)
\widehat{\sigma}_{(\beta',\frac{1}{r_0^2 m}H'_{r_0^2 m^2 n})}),
\end{equation}
\end{NB}
\end{prop}

\begin{proof}
We note that
\begin{equation}
v_0=e^\beta(d H_n+D_n), \; d_n=\frac{(\eta \cdot H_n)}{(H_n^2)}, D_n \in H_n^\perp.
\end{equation}
Since $(H_n^2)=2n$ and
$d_n=\frac{r_0}{(H_n^2)}=\frac{r_0}{2n}$, we get 
\begin{equation}
\frac{m^2(H_n^2)-\chi({\cal O}_X)}{2}>
r_0 l \frac{(\eta \cdot H_n)^2}{(H_n^2)} .
\end{equation}
Proposition \ref{prop:asymptotic} implies that
${\cal E}_{|X \times \{ x' \}} \in {\cal M}_{(\beta,mH_n)}(v_0)$. 
Since
$\widehat{Z}_{(\beta,mH_n)}({\cal E}_{|X \times \{ x' \}}) 
\in \sqrt{-1}{\Bbb R}_{>0}$,
$k_{x'}$ are $\Phi(\widehat{\sigma}_{(\beta,mH_n)})$-stable objects of
phase $\frac{1}{2}$.
By \eqref{eq:elliptic:Phi} and Lemma \ref{lem:geometric},
our claim holds.
\begin{NB}
$$
\phi_{(\beta,mH_n)}({\cal E}_{|X \times \{x'\}})
=\phi_{(\beta',\frac{1}{r_0^2 m}H'_{r_0^2 m^2 n})}(k_{x'})-\frac{1}{2}.
$$
\end{NB}
\end{proof}

\begin{rem}
If $X$ is a Weierstrass elliptic surface and $r_0=1$, then 
a similar claim is proved in \cite{Lo-M} and \cite[Thm. 11.7]{Lo}.

\end{rem}

\begin{rem}
For complex numbers $z,w$ in the upper half plane,
we can also prove that
\begin{equation}
\overline{\Phi}(e^{\beta+zH+w f})=
-r_0 z e^{\beta'+\tfrac{-1}{r_0^2 z}H'+wf'}
\end{equation}
and get a similar claim to Proposition \ref{prop:FMstability} if we set
$\alpha:=\frac{1}{2}+\frac{\log(r_0 z)}{\pi}\sqrt{-1}$.
\begin{NB}
It may be safe to assume that $(z,w)$ is in a neighborhood of $(m \sqrt{-1},mn \sqrt{-1})$.
\end{NB}
\begin{NB}
\begin{equation}
\begin{split}
\overline{\Phi}(e^{\beta+aH+bf+\sqrt{-1}m(H+nf)})
=& \frac{1}{r_0}H' e^{\beta'}-(a+b \sqrt{-1})r e^{\beta'}+(b+mn \sqrt{-1})\tfrac{1}{r_0}\varrho_{X'}
-(a+m \sqrt{-1})(b+mn \sqrt{-1})v_0'\\
=& -r(a+m \sqrt{-1}) 
\exp \left(\beta'+\frac{-a}{r^2(a^2+m^2)}H'+bf'+\sqrt{-1}\frac{m}{r^2(a^2+m^2)}(H'+r^2 (a^2+m^2) nf)' \right).
\end{split}
\end{equation}
\end{NB}
\end{rem}

\begin{NB}
By \cite[Prop. 3.2.1]{MYY:2011:1} (see also \cite[Prop. 6.11]{Y:blowup}), if $H_n$ is ample and
\begin{equation}\label{eq:elliptic:star}
\frac{m^2(H_n^2)}{2}>
1+\frac{(\eta,H_n)^2}{2(H_n^2)},
\end{equation}
\begin{NB2}
$v=e^\beta \eta$.
$\eta=\frac{(\eta \cdot H_n)}{(H_n^2)}H_n+D_n$,
$D_n \in H_n^\perp$.
In particular $a_\beta=0$.

By \cite[6.2]{Y:blowup},
\begin{equation}
\frac{m^2(H_n^2)}{2}>
\frac{(\eta,H_n)}{\delta}
\left(|a_\beta|+1+\frac{(\eta,H_n)^2}{(H_n^2)} \right)
\end{equation}
Since $(f \cdot H)=1$,
\begin{equation}
\frac{m^2(H_n^2)}{2}>
\frac{(\eta,H_n)}{\delta}
\left(|a_\beta|+1+\frac{(\eta,H_n)^2}{(H_n^2)} \right)
\end{equation}

\end{NB2}
then
${\cal M}_{(\beta,mH_n)}(v_0)={\cal M}_{H_n}^\beta(v_0)^{ss}$.
In particular, for $n \gg 0$, we get the claim.

\begin{rem}
Under \eqref{eq:elliptic:star},
$H_{r^2 m^2 n}'$ is ample.
\end{rem}

\end{NB}

\subsection{Fourier-Mukai duality}
In this subsection, we shall treat Fourier-Mukai duality 
for the moduli space of stable 1-dimensional sheaves.
Thus we prove the following result.
\begin{thm}[{cf. \cite[Thm. 3.2.8]{PerverseII}}]\label{thm:rk=0}
Let $\pi:X \to C$ be an elliptic surface
with a fiber $f$.
Let $e^\beta (r_0 f)$ be a primitive and isotropic Mukai vector
with $r_0 >0$ and $H+nf$ an ample divisor with $(H^2)=0$. 
Assume that 
$X':=M_{H+nf}^\beta(e^\beta(r_0 f))$
is a fine moduli space, and let ${\cal E}$ be a universal family.
Then for $\beta',H'$ in \eqref{eq:H'},
we have the following. 
\begin{enumerate}
\item[(1)]
$H'+n' f'$ ($n' \gg 0$) is ample.
\item[(2)]
 ${\cal E}_{|\{ x \} \times X'}$ is $(-\beta')$-twisted stable with respect to
$H'+n' f'$ for all $x \in X$. 
\item[(3)]
\begin{equation}
\begin{matrix}
X & \to & M_{H'+n' f'}^{-\beta'}(e^{-\beta'}(r_0 f'))\\
x & \mapsto & {\cal E}_{|\{ x \} \times X'}
\end{matrix}
\end{equation}
is an isomorphism.
\end{enumerate}
\end{thm} 

\begin{proof}
We note that
\begin{equation}
v_0'=e^{\beta'}(r_0 f')=
e^{\beta'}(d' H_{r_0^2 m^2 n}'+D'),\; d'=\frac{r_0}{2 r_0^2 m^2 n}, D' \in (H_{r_0^2 m^2 n}')^\perp.
\end{equation}
If
\begin{equation}
\frac{2n}{r_0^2}>r_0 l' \frac{1}{2m^2 n}+\frac{\chi({\cal O}_X)}{2},
\end{equation}
then Proposition \ref{prop:asymptotic} implies
$$
{\cal M}_{\widehat{\sigma}_{(\beta',\frac{1}{r_0^2 m}H'_{r_0^2 m^2n})}}(v_0')
={\cal M}_{\frac{1}{r_0^2 m}H'_{r_0^2 m^2n}}^{\beta'-\frac{1}{2}K_{X'}}(v_0')^{ss}.
$$
\begin{NB}
$$
((\tfrac{1}{r_0^2 m} H'_{r_0^2 m^2 n})^2)=\frac{2n}{r_0^2},\; 
\frac{(r_0 f' \cdot H_{r_0^2 m^2 n}')^2}
{({H'_{r_0^2 m^2 n}}^2)}=\frac{1}{2m^2 n}.
$$
\end{NB}

Since 
\begin{equation}
\begin{split}
{\cal M}_{\widehat{\sigma}_{(\beta',\frac{1}{r_0^2 m}H'_{r_0^2 m^2n})}}(v_0')& \cong
{\cal M}_{\widehat{\sigma}_{(\beta,mH_n)}}(\varrho_X)\\
&=\{k_x \mid x \in X \}
\end{split}
\end{equation}
 for
$m^2(H_n^2)>\chi({\cal O}_X)$,
$\Phi(k_x)[-1]={\cal E}_{|\{ x\} \times X'}^{\vee}[1]$
is $\beta'$-twisted stable with respect to $H'_{r_0^2 m^2 n}$
for $n \gg 0$.
Then ${\cal E}_{|\{ x \} \times X'}$ is 
$(-\beta')$-twisted stable with respect to
$H'_{r_0^2 m^2 n}$.
\end{proof}

\section{Preservation of stability}\label{sect:stability}

Let $\Phi$ be the Fourier-Mukai transform in section \ref{sect:FM}.
By Proposition \ref{prop:FMstability}, we have the following.
\begin{prop}\label{prop:FM-stable}
$\Phi$ induces an isomorphism
\begin{equation}
{\cal M}_{H_n}^{\beta-\tfrac{1}{2}K_X}(v)^{ss} \cong {\cal M}_{\widehat{\sigma}_{(\beta',\omega')}}(\overline{\Phi}(v)),
\end{equation}
where $\omega'=\frac{1}{r_0^2 m}H'_{r_0^2 m^2 n}$ and $m \gg n \gg 0$.
\end{prop}
Thus the $\lambda$-stability in \cite{Y:elliptic} is the $\widehat{\sigma}_{(\beta',\omega')}$-stability
in this proposition.

\begin{prop}[{cf. \cite[Thm. 3.15]{Y:7}, \cite[Thm. 3.13]{Y:twist2}}]\label{prop:r=0}
For $v=e^\beta(\xi+a \varrho_X)$ with $a>0$,
$\Phi[-1]=\Phi_{X \to X'}^{{\cal E}^{\vee}[1]}$ induces an isomorphism
\begin{equation}
{\cal M}_{H_n}^{\beta-\tfrac{1}{2}K_X}(v)^{ss} \cong
 {\cal M}_{H'_{n'}}^{\beta'-\frac{1}{2}K_{X'}}(-\overline{\Phi}(v))^{ss}
\end{equation}
for $n,n' \gg 0$.
\end{prop}

\begin{proof}
In the notation of Lemma \ref{lem:coh-FM},
we have
\begin{equation}
\begin{split}
\overline{\Phi}(v)=& -e^{\beta'}(pr_0+a r_0 f'-\tfrac{q}{r_0}\varrho_{X'}-D')\\
=& -
e^{\beta'}(pr_0+d(H_{r_0^2 m^2 n}-(H'-r_0^2 m^2 n f'))-D'-\tfrac{q}{r_0}\varrho_{X'}),
\end{split}
\end{equation}
where
\begin{equation}
d=\frac{ar_0}{2r_0^2 m^2 n}>0,\;
d^2 ({H_{r_0^2 m^2 n}'}^2)=\frac{a^2}{2m^2 n}.
\end{equation}
For $E \in {\cal M}_{H_n}^{\beta-\tfrac{1}{2}K_X}(v)^{ss}$, 
$\widehat{Z}_{(\beta,mH_n)}(E)=-a+(\xi \cdot mH_n)\sqrt{-1}$.
Since 
$\widehat{Z}_{(\beta,mH_n)}({\cal E}_{|X \times \{ x' \}}) \in {\Bbb Z}_{>0}\sqrt{-1}$,
$\phi_{\widehat{\sigma}_{(\beta,mH_n)}}({\cal E}_{|X \times \{ x' \}})=\tfrac{1}{2}<\phi_{\widehat{\sigma}_{(\beta,mH_n)}}(E)$.
Then we get $\phi_{\widehat{\sigma}_{(\beta',\omega')}}(\Phi(E))>1$, where
$\omega'=\frac{1}{r_0^2 m}H'_{r_0^2 m^2 n}$.
Thus $\Phi(E) \in {\cal A}_{(\beta',\omega')}[1]$.
Since
\begin{equation}
({\omega'}^2)=2\frac{n}{r_0^2},
\end{equation}
we can take 
$m \gg n \gg 0$ such that
\begin{equation}
\frac{({\omega'}^2)-\chi({\cal O}_X)}{2}>
\frac{d}{\delta}\frac{d^2 ({H_{r_0^2 m^2 n}'}^2)-2pq}{2}.
\end{equation}
Then Proposition \ref{prop:FM-stable} and Proposition \ref{prop:asymptotic} imply the claim.
\end{proof}

\begin{rem}
\begin{enumerate}
\item[(1)]
If $a<0$, then we also see that $\Phi(E)^{\vee}$ is a twisted stable sheaf, which is treated in  
\cite[Prop. 3.4.5]{PerverseII}.
\item[(2)]
If $r_0=1$, then Lo and Martinez \cite[Thm. 5.12]{Lo-M}
get the same result. 
\end{enumerate}
\end{rem}

\section{Stability conditions on an elliptic K3 surface}\label{sect:K3}

\subsection{Walls for $v=1-\ell \varrho_X$ near the boundary of $\Stab(X)$
associated to the elliptic fibration. }

Let $\pi:X \to C$ be an elliptic K3 surface.
Then Bayer and Macri \cite{BM:2} classified Bridgeland walls.
In this subsection, we shall study walls for the Mukai vector $1-\ell \varrho_X$ near a boundary of $\Stab(X)$.
For a primitive Mukai vector $v$, 
let $P$ be the positive cone in $v^\perp$ and $\overline{P}$ its closure.
Let ${\cal W} \subset \Stab(X)$ be a wall for $v$ defined by a
Mukai vector $u \in H^{*}(X,{\Bbb Z})_{\alg}$.
Then ${\cal W}$ defines a wall $W:=u^\perp$ in $\overline{P}$. 
For $(\beta, \omega) \in \NS(X)_{\Bbb R} \times \Amp(X)_{\Bbb R}$ with 
$(\omega^2)>\chi({\cal O}_X)=2$,
we set
\begin{equation}
\xi(\beta,\omega):=[
\mathrm{Im} (\overline{\langle e^{\beta+\sqrt{-1} \omega},v) \rangle}
e^{\beta+\sqrt{-1} \omega})] \in P/{\Bbb R}_{>0},
\end{equation}
where $[u]:={\Bbb R}_{>0}u$ is the equivalence class of $u \in P$ in 
$P/{\Bbb R}_{>0}$.
We have an identification
\begin{equation}
\theta_v:v^\perp \cong H^2(M_H(v),{\Bbb Z})_{\alg}
\end{equation}
by using the universal family of stable sheaves. 
Under this identification, the boundaries of the movable cone and the nef cones of  
minimal models of $M_H(v)$ are $W=u^\perp$ defining 
walls for $\Stab(X)$. 
For $v=1-\ell \varrho_X$,
we have 
\begin{equation} 
v^\perp={\Bbb Z}\nu+\NS(X),\;\nu:=1+\ell \varrho_X.
\end{equation}
In this case, $[f] \in \overline{P}$ and  
defines a fibration $M_H(v) \to S^{\ell+1} {\Bbb P}^1$.
Let us classify walls near $[f]$.

\begin{lem}\label{lem:f-wall}
Let $W$ be a wall for $v=1-\ell \varrho_X$ in $\overline{P}$ such that $f \in W$.
Then there is a Mukai vector $u=\xi+a \varrho_X$ such that $W=u^\perp$ and one of the following hold:
\begin{enumerate}
\item[(1)]
$\xi \in {\Bbb Z}f$ and $-\ell-1 \leq a \leq 0$ or
\item[(2)]
$\xi$ or $-\xi$ is an effective $(-2)$-divisor in a fiber of $\pi$ and
$-\ell \leq a \leq 0$,
\end{enumerate}
\end{lem}

\begin{proof}
Let $u=r+\xi+a \varrho_X$ $(r,a \in {\Bbb Z}, \xi \in \NS(X))$ be a Mukai 
vector such that $u^\perp =W$. We note that $(v-u)^\perp=u^\perp$.
By \cite{BM:2}, we may assume that 
\begin{enumerate}
\item
$\langle u^2 \rangle=-2$ and 
$0 \leq \langle v,u \rangle \leq \ell$, or
\item
 $\langle u^2 \rangle=0$ and 
$0 < \langle v,u \rangle \leq \ell$, or
\item
$0<\langle u^2 \rangle$ and $2\langle u^2 \rangle+1 \leq \langle v,u \rangle \leq \ell$.
\end{enumerate}
By $f \in u^\perp$, we have
$(\xi,f)=0$. 
Assume that $ra<0$.
Since $\langle v,u \rangle=r\ell-a$, we see that
$\langle v,u \rangle<0$ or $\langle v,u \rangle>\ell$.
Therefore $ra \geq 0$. 
Then we get $\langle u^2 \rangle=(\xi^2)-2ra \leq 0$. Thus we may assume that $\langle u^2 \rangle=0$
or $\langle u^2 \rangle=-2$.

We fitst treat the case where $\langle u^2 \rangle=0$.
In this case $\xi=kf$ and $ra=0$.
If $a=0$, then $\langle v,u \rangle=r\ell$. Hence $r=1$.
In this case, $w:=v-u=-kf-\ell \varrho_X$
satisfies $\langle w^2 \rangle=0$ and $\langle v,w \rangle=\ell$.
If $r=0$, then $\langle v,u \rangle=-a$ implies $0 \geq a \geq -\ell$.
Hence $u=kf+a \varrho_X$ with $0 \geq a \geq -\ell$.

We next treat the case where $\langle u^2 \rangle=-2$.
In this case $(\xi^2)=-2$ and $ra=0$ or $\xi=kf$ and $ra=1$.
For the first case, 
if $a=0$, then we also see that $u=1+\xi$ and $w:=v-w=-\xi-\ell \varrho_X$
satisfies $\langle w^2 \rangle=-2$ and
$\langle v,w \rangle=\ell$.
If $r=0$, then $u=\pm D+a \varrho_X$ with $0 \geq a \geq -\ell$, where
$D$ is an effective $(-2)$-curve in a fiber.
For the second case,
if $r=a=1$, then $\langle v,u \rangle=\ell-1$.
Then $w:=v-u=-kf-(\ell+1)\varrho_X$, $\langle w^2 \rangle=0$ and $\langle v,w \rangle=\ell+1$.
Therefore our claim holds.
\end{proof}

\begin{lem}\label{lem:xi}
Assume that $v=1-\ell \varrho_X$. Then
for $(\beta,\omega) \in \NS(X)_{\Bbb R} \times \Amp(X)_{\Bbb R}$,
\begin{equation}
\xi(\beta,\omega)= [(\beta \cdot \omega)\nu+
(\beta \cdot \omega)\beta+(\ell+\tfrac{(\omega^2)}{2}-\tfrac{(\beta^2)}{2})\omega].
\end{equation}
\end{lem}

\begin{proof}
Since
\begin{equation}
\begin{split}
& \overline{\langle e^{\beta+\sqrt{-1} \omega},v \rangle} e^{\beta+\sqrt{-1}\omega}\\
= & \left(\ell-\frac{(\beta^2)-(\omega^2)}{2}+\sqrt{-1}(\beta \cdot \omega) \right)
(1+(\beta+\sqrt{-1}\omega)+
(\tfrac{(\beta^2)-(\omega^2)}{2}+\sqrt{-1}(\beta \cdot \omega))\varrho_X),\\
\end{split}
\end{equation}
we get
\begin{equation}
\mathrm{Im} (\overline{\langle e^{\beta+\sqrt{-1} \omega}, v \rangle}
e^{\beta+\sqrt{-1} \omega})
= (\beta \cdot \omega)\nu+
(\beta \cdot \omega)\beta+(\ell+\tfrac{(\omega^2)}{2}-\tfrac{(\beta^2)}{2})\omega.
\end{equation}
\end{proof}

\begin{NB}
We set $v:=(r,\xi,a)=\ch(E)(1,0,\tfrac{\chi({\cal O})}{2})$.
Since
\begin{equation}
\begin{split}
& \overline{\langle e^{\beta+\sqrt{-1} \omega},(r,\xi,a) \rangle} e^{\beta+\sqrt{-1}\omega}\\
= & \left((\beta \cdot \xi)-a-r\frac{(\beta^2)-(\omega^2)}{2}+\sqrt{-1}((r\beta- \xi)\cdot \omega) \right)(1,\beta+\sqrt{-1}\omega,
\tfrac{(\beta^2)-(\omega^2)}{2}+\sqrt{-1}(\beta \cdot \omega))\\
\end{split}
\end{equation}
we get
\begin{equation}
\begin{split}
& \mathrm{Im} (\overline{\langle e^{\beta+\sqrt{-1} \omega},(r,\xi,a) \rangle}
e^{\beta+\sqrt{-1} \omega})\\
=& ((\beta-\tfrac{\xi}{r}) \cdot \omega)((r,0,-a)+(0,r\beta,(\beta \cdot \xi)))
+(\tfrac{(\omega^2)}{2}-\tfrac{(\beta^2)}{2}+(\beta \cdot \tfrac{\xi}{r})-\tfrac{a}{r})(0,r\omega,(\omega \cdot \xi))\\
=& ((\beta-\tfrac{\xi}{r}) \cdot \omega)((r,0,-a)+(0,r\beta,(\beta \cdot \xi)))
+(\tfrac{(\omega^2)}{2}-\tfrac{((\beta-\tfrac{\xi}{r})^2)}{2}+\tfrac{\langle v^2 \rangle}{2r^2})(0,r\omega,(\omega \cdot \xi)).
\end{split}
\end{equation}
\end{NB}

\subsection{Relation with walls in Bridgeland stability for an elliptic K3 surface.}

Let $\Phi:{\bf D}(X) \to {\bf D}(X')$ be the Fourier-Mukai transform in section \ref{sect:FM} and assume that
$\Phi(v)=1-\ell \varrho_{X'}$.
In this subsection, let us study chamber structure for $1-\ell \varrho_{X'}$ near $[f']$.
For 
\begin{equation}
\omega':=\frac{1}{r_0^2 m}(H'+r_0^2 m^2 n f')=\frac{1}{r_0^2 m}H'+mn f',
\end{equation}
we have
\begin{equation}
(\beta' \cdot \omega')=\frac{1}{r_0^2 m}(\beta' \cdot H')+mn(\beta' \cdot f')
\end{equation}
and

\begin{equation}
\frac{(\beta' \cdot \omega')}{mn(\ell+\frac{({\omega'}^2)}{2}-\frac{({\beta'}^2)}{2})}=
\frac{\frac{(\beta' \cdot H')}{r_0^2 m^2 n}+(\beta' \cdot f')}{\ell+\frac{n}{r_0^2}-\frac{({\beta'}^2)}{2}}.
\end{equation}
Hence
\begin{equation}\label{eq:behavior}
\xi(\beta',\omega')=
\left[
\frac{\frac{(\beta' \cdot H')}{r_0^2 m^2 n}+(\beta' \cdot f')}{\ell+\frac{n}{r_0^2}-\frac{({\beta'}^2)}{2}}(\nu+\beta')
+\frac{1}{r_0^2 m^2 n}H'+f' \right] \to [f'] \quad (n, m/n \to \infty).
\end{equation}
Thus $\xi(\beta',\omega')$ is an ample divisor of 
$M_{\widehat{\sigma}(\beta',\omega')}(1-\ell \varrho_{X'})$
which is very close to $[f']$.

Let $\Delta$ be a finite polyhedral cone in $\overline{P}$ which contains
$[f'],[H'],-[\nu+\beta']$. 
In a small neighborhood $U$ of $[f']$ in $\Delta$, 
all walls contain $[f']$ (see \cite[Prop. 2.2]{MY} and the paragraph after \cite[Conj. 1.1]{MY}), and hence 
walls are defined by Mukai vectors $u$ in Lemma \ref{lem:f-wall}.
We assume that $(\beta' \cdot f')<0$.
Let us consider a family of stability conditions
$\widehat{\sigma}_{(\beta',t \omega')}$ $(1 \leq t)$ and study the wall crossing.
We note that
$$
\xi(\beta',t\omega')=
\left[
\frac{\frac{(\beta' \cdot H')}{r_0^2 m^2 n}+(\beta' \cdot f')}{\ell+t^2 \frac{n}{r_0^2}-\frac{({\beta'}^2)}{2}}(\nu+\beta')
+\frac{1}{r_0^2 m^2 n}H'+f' \right].
$$
Since $m \gg n \gg 0$, we may assume that
$\frac{(\beta' \cdot H')}{r_0^2 m^2 n}+(\beta' \cdot f')<0$ and
$\xi(\beta',t\omega')$ belongs to the neighborhood $U$.
For a Mukai vector $u=(0,\xi,a)$ in Lemma \ref{lem:f-wall},
$\xi(\beta',t\omega') \in u^\perp$ if and only if 
$$
\frac{\frac{(\beta' \cdot H')}{r_0^2 m^2 n}+(\beta' \cdot f')}{\ell+t^2 \frac{n}{r_0^2}-\frac{({\beta'}^2)}{2}}=
\frac{(H' \cdot \xi)}{r_0^2 m^2 n (a-(\beta' \cdot \xi))}.
$$
Hence $u$ satisfies 
\begin{equation}\label{eq:f-u}
\frac{\frac{(\beta' \cdot H')}{r_0^2 m^2 n}+(\beta' \cdot f')}{\ell+\frac{n}{r_0^2}-\frac{({\beta'}^2)}{2}}<
\frac{(H' \cdot \xi)}{r_0^2 m^2 n (a-(\beta' \cdot \xi))}<0.
\end{equation}
Then the wall crossing along $t \geq 1$ is the same as the wall crossing for $\lambda$-stability 
in \cite[sect. 5]{Y:elliptic}.
\begin{NB}
\begin{lem}
There are finitely many $u$ satisfying \eqref{eq:f-u}.
\end{lem}

\begin{proof}
Let $D_1,D_2,...,D_N$ be representatives of $(-2)$-classes in $f^\perp/{\Bbb Z}f$.
Then we have $\xi=D_i+kf$, where $1 \leq i \leq N$ and $k \in {\Bbb Z}$.
If $(H' \cdot \xi')>0$, then $a-(\beta' \cdot \xi)<0$ implies $k \geq 0$ and 
$-\ell-1 \leq a<(\beta' \cdot D_i)+k(\beta' \cdot f)$ Hence the choice of $k$ is finite.
If $(H' \cdot \xi') < 0$, then $a-(\beta' \cdot \xi)>0$ implies $k \leq 0$ and 
$a>(\beta' \cdot D_i)+k(\beta' \cdot f)$. Hence the choice of $k$ is finite.
\end{proof}
\end{NB}

\begin{rem}
\begin{enumerate}
\item[(1)]
By \eqref{eq:behavior},
$$
\xi(\beta',\omega') \sim 
\left[f'+\frac{(\beta' \cdot f')}{\ell+\frac{n}{r_0^2}-\frac{({\beta'}^2)}{2}}(\nu+\beta') \right],\;  (m \gg 0).
$$
Hence
$\delta+\beta'$ determines the chamber where $\xi(\beta',\omega')$ belongs.
\item[(2)]
If all fibers of $\pi$ are irreducible, then
$u=(0,kf',a)$ with 
$-\ell-1 \leq a<0$ and $a<k(\beta' \cdot f')$.
\end{enumerate}
\end{rem}

\begin{rem}
We also have a similar description of walls for an elliptic abelian surface by results in \cite{movable}.
\end{rem}

\begin{NB}

\section{}

\begin{equation}
\langle e^{\beta+\sqrt{-1}(t^{-1}H+nt f)},e^\beta(r+\xi^\beta+a^\beta) \rangle=
nr-a^\beta+\sqrt{-1}((t^{-1}H+ntf) \cdot \xi).
\end{equation}

\begin{equation}
\begin{split}
& \frac{a^\beta-nr}{nt(\xi^\beta \cdot f)+t^{-1}(\xi^\beta \cdot H)}-
\frac{a_1^\beta-nr_1}{nt(\xi_1^\beta \cdot f)+t^{-1}(\xi_1^\beta \cdot H)}\\
=& \frac{n^2 t(r_1(\xi^\beta \cdot f)-r(\xi_1^\beta \cdot f))+
nt((-a_1^\beta (\xi^\beta \cdot f)-(-a)(\xi_1^\beta \cdot f))}
{(nt(\xi^\beta \cdot f)+t^{-1}(\xi^\beta \cdot H))(nt(\xi_1^\beta \cdot f)+t^{-1}(\xi_1^\beta \cdot H))}\\
& +\frac{n t^{-1}(r_1(\xi^\beta \cdot H)-r(\xi_1^\beta \cdot H))+
t^{-1}((-a_1^\beta (\xi^\beta \cdot H)-(-a)(\xi_1^\beta \cdot H))}
{(nt(\xi^\beta \cdot f)+t^{-1}(\xi^\beta \cdot H))(nt(\xi_1^\beta \cdot f)+t^{-1}(\xi_1^\beta \cdot H))}
\end{split}
\end{equation}

For $t \gg n \gg 0$,

We set $\omega:=t^{-1}H+tn f$.

Let $E$ be a $\widehat{\sigma}_{(\beta,\omega)}$-semi-stable object with $v(E)=e^\beta(r+\xi+a\varrho_X)$
with $(\xi \cdot f)>0$.
Assume that $H^{-1}(E) \ne 0$.
We set $v(H^{-1}(E)[1]):=e^\beta(r_1+\xi_1 +a_1 \varrho_X)$.
Then $r_1<0$ and $(\xi_1 \cdot (H+t^2 n f)) \geq 0$ for all $t \gg n \gg 0$.
In particular $(\xi_1 \cdot f) \geq 0$.
Then $r_1(\xi \cdot f)-r(\xi_1 \cdot f)<0$, which implies $E$ is not semi-stable.
Therefore $E$ is a sheaf.

For an exact sequence 
\begin{equation}
0 \to E_1 \to E \to E_2 \to 0
\end{equation}
in $\Coh(X)$, we set 
$v(E_i):=e^\beta(r_i+\xi_i +a_i \varrho_X)$.
Since $E \in {\cal A}_{(\beta,\omega)}$ for $t \gg n \gg 0$,
$(\xi_2 \cdot(H+nt^2 f))>0$ for $t \gg n \gg 0$.
Thus $(\xi_2 \cdot f) \geq 0$.
If $E$ is not $f$-semi-stable, then by taking the Harder-Narashimhan filtration
of $E$, 
we may assume that $E_1 \in {\cal A}_{(\beta,\omega)}$,
$r_2>0$ and $r_2 (\xi \cdot f)-r(\xi_2 \cdot f) \geq 0$.
By $\widehat{\sigma}_{(\beta,\omega)}$-semi-stability of $E$,
$r_2 (\xi \cdot f)-r(\xi_2 \cdot f) \leq 0$.
Therefore $r_2 (\xi \cdot f)-r(\xi_2 \cdot f)=0$.
Thus $E$ is $f$-semi-stable.

Assume that $\pi:X \to C$ is an elliptic abelian surface.
Then $(0,f,0)$ defines a fibration $M_H(v) \to S^{\ell} C$.

Mukai vectors $u$ such that
$\langle u^2 \rangle=0$ and 
$0 < \langle v,u \rangle \leq \ell$,
$0<\langle u^2 \rangle$ and $2\langle u^2 \rangle+1 \leq \langle v,u \rangle \leq \ell$.

Let $X=C_1 \times C_2$ be a product of elliptic curves.

\begin{equation}
\beta:=xf+yg,\;\omega:=zf+wg. 
\end{equation}

\begin{equation}
e^{\beta+\sqrt{-1}\omega} \Longleftrightarrow 
\begin{pmatrix}
1 & x+\sqrt{-1} z\\
y+\sqrt{-1} w & (x+\sqrt{-1}z)(x+\sqrt{-1} w)
\end{pmatrix}
=
\begin{pmatrix}
1\\
y+\sqrt{-1} w 
\end{pmatrix}
\begin{pmatrix}
1 & x+\sqrt{-1} z
\end{pmatrix}
\end{equation}

\begin{equation}
\begin{split}
& \mathrm{Im} (\overline{\langle e^{\beta+\sqrt{-1} \omega},(1,0,-\ell) \rangle}
e^{\beta+\sqrt{-1} \omega})\\
=& (xw+yz)(1,0,\ell)+(0,(\ell z+z^2 w+x^2 w)f+(\ell w+zw^2+y^2 z)g,0)
\end{split}
\end{equation}

We consider a Fourier-Mukai transform
on $C_1$ associated to 
\begin{equation}
A:=
\begin{pmatrix}
r & r' \\
d & d'
\end{pmatrix} \in \SL(2,{\Bbb Z}).
\end{equation}
Then $\Phi$ induces an autoequivalence of $X$:
\begin{equation}
A \begin{pmatrix}
1 & q\\
p & pq
\end{pmatrix}
=A
\begin{pmatrix}
1\\
p
\end{pmatrix}
\begin{pmatrix}
1 & q
\end{pmatrix}
=(a+bp)
\begin{pmatrix}
1\\
\frac{c+dp}{a+bp}
\end{pmatrix}
\begin{pmatrix}
1 & q
\end{pmatrix}.
\end{equation}

\begin{equation}
\Phi(e^{pg+qf})=(a+bp) e^{\frac{c+dp}{a+bp}g+qf}.
\end{equation}

\begin{equation}
\beta+\sqrt{-1} \omega=(x+\sqrt{-1}z) f+(y+\sqrt{-1} w)g=
qf+pg.
\end{equation}

Assume that $w \gg 0$ and $z/w \gg 0$.

\begin{equation}
\begin{split}
\frac{c+dp}{a+bp}g+qf= & \left(\frac{(a+by)(c+dy)+bd w^2}{b^2 w^2+(a+by)^2}g+xf \right)+
\sqrt{-1} \left(\frac{w}{b^2 w^2+(a+by)^2}g+zf \right)\\
\sim & \frac{d}{b}g+xf+\sqrt{-1} \left(\frac{1}{b^2 w}g+zf \right).
\end{split}
\end{equation}

We set $v_0:=(0,r' f,-r)=e^\beta (r' f)$, where $\beta:=\frac{-r}{r'}\sigma$.
Then $\Phi(v_0)=-(0,0,1)$.
We set $\beta':=\frac{d'}{r'}\sigma$.
Then $v_0'=e^{\beta'} (r' f)$ and 
$\Phi(\sigma)=r'e^{\beta'}$.

We see that $\delta+\beta' \in (0,r' f,d')^\perp$.

For $E \in {\bf D}(X)$,
$r(E)$ is the rank of $E$ and $d(E):=(c_1(E) \cdot f)$ is the relative degree of $E$.

Let $\Phi:{\bf D}(X) \to {\bf D}(Y)$ be a relative Fourier-Mukai transform
such that
\begin{equation}
\begin{pmatrix}
r(\Phi(E))\\
d(\Phi(E))
\end{pmatrix}
=
\begin{pmatrix}
c & a\\
d & b
\end{pmatrix}
\begin{pmatrix}
r(E)\\
d(E)
\end{pmatrix}
\end{equation}
Then \cite[4.3]{Br:1} implies
$\Phi((0,pf,q))=(0,(cp+aq)f,dp+bq)$.

\begin{NB2}
Let $\pi_X':X \times_C Y \to X$ and $\pi_Y':X \times_C Y \to Y$ be projections. 

\begin{lem}
$\ch \Phi(E(\lambda f))=e^{\lambda f}\ch \Phi(E)$.
\end{lem}

\begin{proof}
Let ${\bf P}$ be the kernel.
Then we see that
\begin{equation}
\begin{split}
\ch ({\bf P}) \otimes {\pi_Y'}^*( e^{\lambda \pi_X^{-1}(c)})=& 
\ch ({\bf P})+\lambda \ch ({\bf P}) (\pi_X^{-1}(c) \times \pi_Y^{-1}(c))\\
 \ch ({\bf P}) \otimes {\pi_X'}^*( e^{\lambda \pi_Y^{-1}(c)})=& 
\ch ({\bf P})+\lambda \ch ({\bf P}) (\pi_X^{-1}(c) \times \pi_Y^{-1}(c)),
\end{split}
\end{equation}
which implies
\begin{equation}
\ch ({\bf P}) \otimes {\pi_Y}^*( e^{\lambda \pi_X^{-1}(c)})=
\ch ({\bf P}) \otimes {\pi_X}^*( e^{\lambda \pi_Y^{-1}(c)}).
\end{equation}
Therefore the claim holds.
\end{proof}

\end{NB2}

\section{Wall crossing for moduli of 1-dimensional sheaves}

\begin{equation}
\frac{\chi-(D \cdot \eta)}{(D \cdot H)}=\frac{\chi'-(D' \cdot \eta)}{(D' \cdot H)}
\iff \chi (D' \cdot H)-\chi' (D \cdot H)=(D' \cdot H)(D \cdot \eta)-(D \cdot H)(D' \cdot \eta)
\end{equation}

If $H=H_0+nf$ and $\eta=xH_0$, then
\begin{equation}
x=\frac{\chi(D' \cdot f)-\chi' (D \cdot f)}{(D \cdot H_0)(D' \cdot f)-(D' \cdot H_0)(D \cdot f)}+
\frac{1}{n}\frac{\chi(D' \cdot H_0)-\chi' (D \cdot H_0)}{(D \cdot H_0)(D' \cdot f)-(D' \cdot H_0)(D \cdot f)}.
\end{equation}

\subsection{An elliptic ruled surface}

Let $\varpi:X \to C$ be an elliptic ruled surface with a minimal section $\sigma$
with $(\sigma^2)=1$. Let $g$ be a fiber of $\varpi$.

\begin{lem}
Let $D$ be an effective divisor.
Then 
\begin{enumerate}
\item
$D$ is nef. 
\item
$(D^2)=0$ if and only if $D \in {\Bbb Z}_{>0} (2\sigma-g)$ or $D \in {\Bbb Z}_{>0} g$.
\item
$(D^2)=1$ if and only if $D=\sigma$.
\end{enumerate} 
\end{lem}

\begin{lem}
Let $D_1,D_2$ be effective divisors.
\begin{enumerate}
\item
$(D_1,D_2)=0$ if and only if
$D_1, D_2 \in {\Bbb Z}_{>0} (2\sigma-g)$ or $D_1,D_2 \in {\Bbb Z}_{>0} g$.
\item
$(D_1,D_2)=1$ if and only if one of the following holds.
\begin{enumerate}
\item
$D_1=D_2=\sigma$.
\item
$\{ D_1,D_2 \}=\{2\sigma-g,\sigma+n(2\sigma-g)\}$, $n \in {\Bbb Z}_{\geq 0}$.
\item
$\{D_1,D_2 \}=\{g,\sigma+ng \}$,  $n \in {\Bbb Z}_{\geq 0}$.
\end{enumerate}
\end{enumerate}
\end{lem}

\begin{proof}
Assume that $(D_1 \cdot D_2)=1$ and $(D_1^2)=0$.
Then $D_1=2\sigma-g,g$.
If $D_1=2\sigma-g$, then 
$(\sigma \cdot (2\sigma-g))=1$ implies
$D_2-\sigma=n(2\sigma-g)$ $(n \in {\Bbb Z})$.
If $\sigma+\xi$ is effective, then $n \geq 0$.
Therefore $D_2=\sigma+n(2\sigma-g)$ $(n \in {\Bbb Z}_{\geq 0})$.
 
If $D_1=g$, then we also see that $D_2=\sigma+ng$ $(n \in {\Bbb Z}_{\geq 0})$.
\end{proof}
 
\end{NB}

\end{document}